\documentclass[11pt, reqno]{amsart}
\usepackage{amssymb, hyperref}
\usepackage{amsthm}
\usepackage[all,2cell]{xy}
\usepackage{verbatim}
\usepackage{xcolor}
\usepackage{hyperref}
\usepackage{mathrsfs}
\usepackage{csquotes}
\usepackage{epigraph}

\numberwithin{equation}{section}
\newtheorem{theorem}{Theorem}[section]
\newtheorem{corollary}[theorem]{Corollary}
\newtheorem{definition}[theorem]{Definition}
\newtheorem{lemma}[theorem]{Lemma}

\newtheorem{proposition}[theorem]{Proposition}
\newtheorem{remark}[theorem]{Remark}
\newtheorem*{theorem*}{Theorem}
\newtheorem*{corollary*}{\bf Corollary}

\theoremstyle{remark}
\newtheorem{example}[theorem]{Example}

\title[Gromov width of Bott-Samelson varieties]{The Gromov width of Bott-Samelson varieties}

\author[B.~N.~Chary]{Narasimha Chary Bonala}
\address{Ruhr-Universit\"at Bochum, Fakult\"at f\"ur Mathematik, D-44780 Bochum, Germany}
\email{Narasimha.Bonala@rub.de and stephanie.cupit@rub.de}
\author[S.~Cupit-Foutou]{St\'ephanie Cupit-Foutou}
\thanks{This research was supported by the  CRC/TRR 191 “Symplectic Structures in Geometry, Algebra and Dynamics” of the Deutsche Forschungsgemeinschaft.}

\begin{document}

\begin{abstract}  We prove that the Gromov width of any Bott-Samelson variety
associated to a reduced expression and equipped with a rational K\"ahler form equals the symplectic area of a minimal curve.
From this, we derive an estimate for the Seshadri constants of ample line bundles on Bott-Samelson varieties.
\end{abstract}

\maketitle

\section{Introduction} 
The Gromov width of a $2n$-dimensional symplectic manifold is the largest capacity $a$ for which a ball in $\mathbb R^{2n}$ of radius $\sqrt{a/\pi}$, equipped with the standard symplectic form, can be symplectically embedded in this manifold. By Darboux Theorem, this symplectic invariant is a positive number; it originates in the work of Gromov's while proving his celebrated non-squeezing theorem (\cite{Gro}). There have been many works dedicated to the computation or estimates of the Gromov width of symplectic manifolds (see e.g.~\cite{biran, kartol,LMZ,Cas,FLP,HLS} and references therein). 

In this paper, we compute the Gromov width of the Bott-Samelson varieties which are  natural desingularizations of complex Schubert varieties. In particular, we extend the results on the Gromov width of rational coadjoint orbits of connected simple compact Lie groups obtained in~\cite{FLP}.
Gromov widths are closely related to Seshadri constants of line bundles (see~\cite{MDP}). The latter have been busily investigated in algebraic geometry as a measure of local positivity (see e.g. Chap.5 in~\cite{Laz04}).
From our computation of the Gromov width, we derive an estimate for the Seshadri constants of ample line bundles on Bott-Samelson varieties.

Our computation of the Gromov width is performed in two steps. 
First, by using Gromov's $J$-holomorphic curves techniques applied to minimal curves, we prove that the Gromov width of the Bott-Samelson varieties we consider is bounded above by the symplectic areas of the minimal curves of these varieties (Theorem~\ref{thm:main-upper}).
Secondly, to obtain a lower bound, we apply embedding methods: symplectic embeddings of balls into symplectic toric manifolds can be derived from embeddings of simplices into the momentum images of the latter manifolds;
this method can be carried out to more general projective manifolds via toric degenerations associated to Newton-Okounkov bodies (see \cite{Kav, Pab18}). 
 This approach has been already followed by several authors; see e.g.~\cite{Pab18} for some review.
This method applied to concrete simplices included in  Newton-Okounkov bodies unimodular to generalized string polytopes 
(as defined in~\cite{Fujita}) enables us to get a lower bound for the Gromov width in case the Bott-Samelson varieties are equipped with a positive integral $2$-form (Theorem~\ref{thm:main-lower}).
Finally, by making use of Brion--Kannan's characterization (\cite{BK21}) of the minimal curves of Bott-Samelson varieties that are birational to Schubert varieties, we show that the minimum symplectic area among such curves equals the size of one of the simplices alluded above. From this, we can conclude that the lower and upper bounds we obtain do coincide (Corollary~\ref{cor:main}).

To state our results more explicitly, let us set up some notation.
 Let $G$ be a connected semisimple complex algebraic group of rank $n$. Fix a maximal torus $T$ and a Borel subgroup $B\subset G$ containing $T$. 
Let $\alpha_1, \ldots, \alpha_n$  and $\alpha_1^{\vee}, \ldots, \alpha_n^{\vee}$ be the corresponding simple roots and simple coroots respectively. The latter form a basis of the cocharacter lattice $\Xi^*(T)$ with dual basis consisting of the fundamental weights $\varpi_1,\ldots,\varpi_n$.
Let $W$ be the Weyl group of $(G,T)$ and $s_i\in W$ be the reflection corresponding to the simple root $\alpha_i$. 
By $P_{i}$ we denote the minimal parabolic subgroup of $G$ generated by $B$ and any representative $\dot s_i$ of $s_i$.

Given a sequence of simple roots ${\bf i}= (\alpha_{i_1}, \ldots,\alpha_{i_r})$, the corresponding
{\it Bott-Samelson variety} $Z_{\bf i}$ is the quotient
$$
Z_{\bf i}:= (P_{i_1}\times \cdots \times P_{i_r})/B^r
$$
where $B^r$ acts on $P_{i_1}\times \cdots \times P_{i_r}$ by 
$$
(p_1, \ldots, p_r)\cdot (b_1, \ldots, b_r):=(p_1b_1, b_1^{-1}p_2b_2, \ldots, b_{r-1}^{-1}p_rb_r)
$$
for $p_j\in P_{i_j}$ and $b_j\in B$ for all $1\leq j\leq r$. 

Let $w=s_{i_1}s_{i_2}\cdots s_{i_r}$.
Throughout this paper, we assume that 
this is a reduced decomposition of $w\in W$. 
In this case, the corresponding Bott-Samelson variety is a desingularization of the Schubert variety in $G/B$ associated to $w$.

In Section~\ref{sec:BS-varieties}, the reader can find some material on Bott-Samelson varieties and results of~\cite{BK21} on minimal curves.

In Section~\ref{sec:up},
we prove that the Gromov width $w_G(Z_{\bf i}, \omega)$ of $Z_{\bf i}$ equipped with any K\"ahler $2$-form $\omega$ is bounded from above by the symplectic area $\int_C \omega$ of any minimal curve $C$ of $Z_{\bf i}$. The latter is denoted by $\omega([C])$ below. 

\begin{theorem}\label{thm:main-upper}
Let $[\omega]\in H^2(Z_{\bf i},\mathbb R)$ be K\"ahler.
Then 
$$
w_G(Z_{\bf i}, \omega)\leq\min\{\omega([C]): \mbox{$C$ minimal curve of $Z_{\bf i}$}\}, $$
\end{theorem}

In Section~\ref{sec:lower}, we show the following theorem (and Corollary~\ref{cor:coadorbit}) giving a lower bound for the Gromov width.

\begin{theorem}\label{thm:main-lower}
Let ${\bf m}=(m_1, \ldots, m_r)\in\mathbb Q_{>0}^r$ and $\omega_{\bf m}$ be the associated rational K\"ahler $2$-form of the Bott-Samelson variety $Z_{\mathbf i}$.
Then 
\begin{eqnarray}
w_G(Z_{\bf i}, \omega_{\bf m}) &\geq & 
\min\big\{ 
m_j+\sum_{j<\ell\leq r} m_\ell\langle \varpi_{i_\ell}, s_{i_\ell}\cdots s_{i_{j+1}}(\alpha_{i_{j}}^{\vee})\rangle : 1\leq j\leq r
\big\}
\\
& \geq & \min\{\omega_{{\bf m}}([C]): \mbox{$T$-stable curve $C$ of $Z_{\bf i}$ containing $z_0$}\}
\\
& \geq & \min\{\omega_{{\bf m}}([C]): \mbox{$C$ minimal curve of $Z_{\bf i}$}\}.\label{ineq:geo}
\end{eqnarray}
\end{theorem}

Theorem~\ref{thm:main-lower} extends the results obtained by Fang-Littelmann-Pabiniak in~\cite{FLP} concerning the lower bound of the Gromov-width of rational coadjoint orbits equipped with the Kirillov-Kostant-Souriau symplectic form. Namely, we recover

\begin{corollary}[{Fang-Littelmann-Pabiniak}]\label{cor:coadorbit}
Let $K$ be the connected compact Lie group such that $G=K^\mathbb C$.
Given $\lambda\in\Xi^*(T)\otimes_{\mathbb Z}\mathbb Q$ in the Weyl chamber defined by $B$, let $O_\lambda$ be the coadjoint orbit $K\lambda$ equipped with the Kirillov-Kostant-Souriau form.
Then
$$
w_G(O_\lambda)\geq\mathrm{min}\big\{\lambda(\alpha_j^\vee):\lambda(\alpha_j^\vee)\neq 0, 1\leq j\leq n\big\}.
$$
\end{corollary}

\begin{remark}
It is an open conjecture of Karshon and Tolman that the Gromov width of a coadjoint orbit of a connected compact simple Lie group is precisely the lower bound given in the above corollary; see~\cite{AHLL,HL} for the state of the art on this conjecture.
\end{remark}

As a straightforward consequence of Theorem~\ref{thm:main-upper} and Theorem~\ref{thm:main-lower}, we get the following:

\begin{corollary}\label{cor:main}
Let ${\bf m}=(m_1, \ldots, m_r)\in\mathbb Q_{>0}^r$ and $\omega_{\bf m}$ be the associated rational K\"ahler $2$-form of the Bott-Samelson variety $Z_{\mathbf i}$.
Then 
\begin{equation*}
\begin{array}{lll}
w_G(Z_{\bf i}, \omega_{\bf m}) & =& \min\big\{ 
m_j+\sum_{j<\ell\leq r} m_\ell\langle \varpi_{i_\ell}, s_{i_\ell}\cdots s_{i_{j+1}}(\alpha_{i_{j}}^{\vee})\rangle : 1\leq j\leq r
\big\}\\
                       & =   & \mathrm{min}\{\omega_{\bf m}([C]): \mbox{$C$ minimal curve of $Z_{\bf i}$}\}.
\end{array}
\end{equation*}
\end{corollary}

As previously mentioned, one can estimate Seshadri constants by Gromov widths.
Given a projective variety $X$ together with an ample line bundle $\mathcal L$ on $X$,
the Seshadri constant $\varepsilon(X,\mathcal L,x)$ of $\mathcal L$ at some point $x\in X$ is defined as
the infimum of $\mathcal L\cdot C/\mathrm{mult}_x C$ taken over all irreducible and reduced curves $C$ on $X$ passing through $x$. Here $\mathrm{mult}_x C$ stands for the multiplicity of $C$ at $x$.
As shown in~\cite[Proposition 2.6.1]{BC01}, in case of an ample line bundle $\mathcal L$, the Seshadri constant $\varepsilon(X,\mathcal L,x)$ is upper bounded by the Gromov width of $X$ equipped with the Fubini-Study form associated to $\mathcal L$. Together with Corollary~\ref{cor:main}, this yields\footnote{While this paper was being reviewed, Biswas, Hanumanthu and Kannan computed Seshadri constants of equivariant bundles on Bott-Samelson varieties at some points (\cite{BHK}).}

\begin{corollary}\label{cor:Seshadri}
Let $\bf{m}\in\mathbb Z^r_{>0}$.
Then the following inequality holds for the Seshadri constants of the Bott-Samelson variety $Z_{\bf{i}}$ at any point $x\in Z_{\bf{i}}$
$$
\varepsilon(Z_{\bf{i}},\omega_{\bf{m}},x)\leq   \mathrm{min}\{\omega_{\bf m}([C]): \mbox{$C$ minimal curve of $Z_{\bf i}$}\}.
$$
\end{corollary}

\begin{remark}
Lower bounds of Seshadri constants can be derived also by embedding methods (see e.g.~\cite{ito}). 
This will be addressed in a forthcoming work.  
\end{remark}

We conclude our work by further relating our results on Gromov widths to previous ones.
In Section~\ref{sec:toric-BS}, we consider the polarized Bott-Samelson varieties which can be degenerated into polarized Bott (toric) manifolds. These toric degenerations (and their underlying combinatorics) were thoroughly studied in~\cite{GK,Pas,HY}. 
The Gromov widths of polarized generalized Bott manifolds are computed in~\cite{HLS}\footnote{As a side result, we prove 
that the Gromov widths of generalized Bott manifolds can be expressed as the symplectic areas of minimal curves of these manifolds (Proposition~\ref{prop:bott-case-reformul}).}.
These results combined altogether enable us to recover Corollary~\ref{cor:main} in this very setting (Corollary~\ref{cor:caseline}); in particular, we get that the Gromov width of the Bott-Samelson varieties under consideration equals the symplectic area of a line. 
This method can not be applied further to any Bott-Samelson variety to recover fully Corollary~\ref{cor:main} as Example~\ref{exple} shows.

\section{Bott-Samelson varieties}\label{sec:BS-varieties}

For the sake of simplicity, we identify the set of simple roots of $(G,B,T)$ with the set  $I=\{1, 2, \ldots, n\}$. 
For $w\in W$, an expression $\bf i$ of $w$ is a word $(i_1, i_2, \ldots, i_r)\in I^r$ such that $w=s_{i_1}\cdots s_{i_r}$. 
Recall that an expression $\bf i$ of $w$ is called reduced whenever the number $r$ is minimal. 

Below we freely recall a few basic facts about Bott-Samelson varieties; for more details, see~\cite{demazure,LT}.

Given a (not necessarily reduced) expression $ {\bf i}=(i_1,\ldots,i_r)$ of $w\in W$,
recall the definition of the Bott-Samelson variety $Z_{\bf i}$ stated in the introduction. 
The variety $Z_{\bf i}$ is smooth, projective  of dimension $r$. 
The left multiplication of $P_{i_1}$ on the first factor makes $Z_{\bf i}$ into a $P_{i_1}$-variety equipped with an $P_{i_1}$-equivariant morphism
$$
\pi_{r}: Z_{\bf i}\longrightarrow G/B, \quad
[(p_1, \ldots, p_r)]\longmapsto p_1\cdots p_rB.
$$

We can realise $Z_{\bf i}$ as an iterated $\mathbb P^1$-bundle. 
Specifically, let $w'=s_{i_1}\cdots s_{i_{r-1}}$ and ${\bf i'}=(i_1, \ldots, i_{r-1})$. Let $f: G/B \longrightarrow G/P_{i_r}$ be the map given by $gB \mapsto gP_{i_r}$, and let $p_{\bf i'}: Z_{\bf i'}\longrightarrow G/P_{i_r}$ be the map given by $[(p_1, \ldots, p_{r-1})]\mapsto p_1\cdots p_{r-1}P_{i_{r}}$. We have the following commutative diagram:
$$
 \xymatrix{Z_{\bf i}=Z_{\bf i'}\times_{G/P_{i_r}}G/B  \ar[rrr]^{\pi_{r}} \ar[d]^{f_{r-1}} &&& G/B \ar[d]^{f}\\
 Z_{\bf i'}\ar[rrr]_{p_{\bf i'}} &&& G/P_{i_r}}~.
 $$

For $1\leq j \leq r$, let $Z_{\bf i}(j)$ denote the Bott-Samelson variety associated to the sub-expression $(i_1, \ldots, i_j)$ of ${\bf i}=(i_1, \ldots, i_r)$. 

Consider the natural morphisms
$$
\pi_j:Z_{\bf i}(j) \longrightarrow G/B \quad
\mbox{and}\quad f_j: Z_{\bf i} \longrightarrow Z_{\bf i}(j).
$$ 
Explicitly, $f_j$ maps $[(p_1, \ldots, p_r)]$ to $[(p_1, \ldots, p_j)]$.

Henceforth, we assume the expression $\bf i$ is reduced. 
Let $X(s_{i_1}\cdots s_{i_j})$ be the Schubert variety associated to the Wel group element $s_{i_1}\cdots s_{i_j}$.
Then 
$$
\pi_j(Z_{\bf i}(j))=X(s_{i_1}\cdots s_{i_j})
$$
and $Z_{\bf i}(j)$ is a desingularization of $X(s_{i_1}\cdots s_{i_j})$.

The Bott-Samelson variety $Z_{\bf i}$ is equipped with a base point that is
$$
z_0=[(\dot s_{i_1}, \ldots, \dot s_{i_r})].
$$ 
Moreover, we have $\pi_r(z_0)=w x_0$
with $x_0$ being the base point of $G/B$ and  $\pi_r$ yields an isomorphism between the $B$-orbits
\begin{equation}\label{eq:B-orbits}
Bz_0\simeq Bwx_0
\end{equation}
(see e.g. \cite[\S II.13.6]{Jan03} for details).
In particular, $z_0$ is fixed by $T$.

For $1\leq j\leq r$, the line bundle $\mathcal L_j$ on $Z_{\bf i}$ is defined as 
$$
\mathcal L_j:=f_j^*\pi_j^*\mathcal L_{G/B}(\varpi_{i_j}),
$$
where $\mathcal L_{G/B}(\varpi_{i_j})$ denotes the line bundle on $G/B$ associated to $\varpi_{i_j}$. 
For ${\bf m}=(m_1, \ldots, m_r)\in \mathbb Z^r$, 
we set 
$$
\mathcal L_{\bf i, m}:=\mathcal L_1^{m_1}\otimes \cdots \otimes \mathcal L_r^{m_r}.
$$

\begin{theorem}[{\cite[Theorem 3.1]{LT}}]\label{thm:LT}\
\begin{enumerate}
\item The isomorphism classes of $\mathcal L_1, \ldots, \mathcal L_r$ form a basis of $\mathrm{Pic}(Z_{\bf i})$. 
   In particular, the map $\mathbb Z^r\longrightarrow \mathrm{Pic}(Z_{\bf i}),\, {\bf m} \mapsto \mathcal L_{\bf i, m}$ is an isomorphism of groups.
    \item  The line bundle $\mathcal L_{\bf i, m}$ is (very) ample if and only if $m_j>0$ for all $j$.
    \item  The line bundle $\mathcal L_{\bf i, m}$ is generated by its global sections if and only if $m_j\geq 0$ for all $j$.
\end{enumerate}
\end{theorem}

In the remainder of this section, we collect some results on the $T$-stable curves and the minimal rational curves on Bott-Samelson varieties from \cite{BK21}.

First, let us briefly recall the notion of rational curves on any projective variety $X$.
For more details, we refer to \cite[Chapter II.2]{Kol}.

Let $\mathrm{RatCurves}(X)$ denote the normalization of the space of rational curves on $X$. 
Every irreducible component $\mathcal K$ of $\mathrm{RatCurves}(X)$ is a (normal) quasi-projective variety equipped with a quasi-finite morphism to the Chow variety of $X$; the image consists of the Chow points of irreducible, generically reduced rational curves. 
There exists a universal family $p :\mathcal U \to \mathcal K$ and a projection $\mu :\mathcal U \to X$.
For any $x\in X$, let $\mathcal U_x=\mu^{-1}(x)$ and $\mathcal K_x=p(\mathcal U_x$).
If  $\mathcal K_x$ is non-empty and projective for a general point $x\in X$ then $\mathcal K$ is called a {\it family of minimal rational curves} and any member of $\mathcal K$ is called a {\it minimal rational curve}.

We now state the main properties and the characterization of minimal curves in case of Bott-Samelson varieties associated to a reduced expression.

Given a reduced word ${\bf i}=(i_1, \ldots, i_r)$ of $w\in W$, 
let 
\begin{eqnarray}\label{eq:def-beta_i}
\beta_1:=\alpha_{i_1},\quad \beta_2:=s_{i_1}(\alpha_{i_2}), \ldots, \quad \beta_r:=s_{i_1}\cdots s_{i_{r-1}}(\alpha_{i_r}).
\end{eqnarray} 
 
Note that 
\begin{eqnarray}\label{eq:beta_i}
R^+\cap w(R^-)=\{\beta_1, \ldots, \beta_r\}
\end{eqnarray} 
where $R^+$ and $R^-$ denote the sets of positive and negative roots corresponding to $B$, respectively.

For any $1\leq j\leq r$, consider the following orbit closure in $Z_{\bf i}$
$$
C_j=\overline{U_{\beta_j}z_0}
$$
with $U_{\beta_j}$ being the root subgroup of $G$ corresponding to $\beta_j$. 

\begin{proposition}[{\cite[Lemma 4.1]{BK21}}]\label{prop:Tcurves}
Assume $Z_{\bf i}$ is a Bott-Samelson variety associated to a reduced expression ${\bf i}$.
\begin{enumerate}
    \item The $T$-stable curves in $Z_{\bf i}$ through $z_0$ are exactly the curves $C_j$.
    \item 
    Every curve $C_j$ is isomorphic to the projective line $ \mathbb P^1$.
    \item For $1\leq j\leq r$ and $1\leq k\leq r$, we have 
$$
\mathcal L_{k}\cdot C_j=\begin{cases}
0 & ~if~ j>k\\
\langle \varpi_{i_k}, 
s_{i_k}\cdots s_{i_{j+1}}
(\alpha_{i_{j}}^{\vee}) \rangle & ~if~j\leq k
\end{cases}~.
$$ 
\item 
Let $K_{Z_{\bf i}}$ be the canonical line bundle on $Z_{\bf i}$. Then we have 
$$
-K_{Z_{\bf i}}\cdot C_j= 
\sum_{i=1}^n \langle\varpi_{i}, s_{i_r}\cdots s_{i_{j+1}}(\alpha_{i_j}^{\vee})\rangle+1.
$$
\end{enumerate}
\end{proposition}

The following result gives a description of the minimal rational curves in $Z_{\bf i}$.

\begin{theorem}[{\cite[Theorem 4.3]{BK21}}]\label{thm:BKT}
Assume $Z_{\bf i}$ is a Bott-Samelson variety associated to a reduced expression ${\bf i}$. 
\begin{enumerate}
   \item Every minimal family $\mathcal K$ on $Z_{\bf i}$ satisfies $\mathcal K_{z_0}=\{C_j\}$ for some $1\leq j\leq r$.
    \item 
The minimal rational curves in $Z_{\bf i}$ through $z_0$ are exactly those $C_j$ such that $s_{i_r}\cdots s_{i_{j+1}}(\alpha_{i_j})$ is a simple root.
\end{enumerate}
\end{theorem}

\begin{corollary}\label{cor:dominant-minimal}
Let $(Z_{\bf i},\mathcal L_{{\bf i,m}})$ be a Bott-Samelson variety equipped with an ample line bundle.
Assume ${\bf i}$ is a reduced expression. Then
\begin{eqnarray*}
\min\{\mathcal L_{{\bf i,m}}\cdot C: \mbox{$C$ minimal curve }\} &= & \min\{\mathcal L_{{\bf i,m}}\cdot C: \mbox{$T$-stable curve $C\ni z_0$}\}\\
& = &
\min\{m_j+\sum_{j<k\leq r}m_k\langle\varpi_{i_k},s_{i_k}\ldots s_{i_{j+1}}(\alpha_{i_j}^\vee)\rangle: 1\leq j\leq r\}.
\end{eqnarray*}
\end{corollary}

\begin{proof}
Note that the curves containing the base point $z_0$ are dominant and recall that the minimal curves of $Z_{\bf i}$ containing $z_0$ are all $T$-stable (Theorem~\ref{thm:BKT}(1)).
Besides, the degrees of all curves in the same minimal family w.r.t. a given polarization  are equal and  thanks to~\cite[Theorem IV.2.4]{Kol}, we know that a family of rational curves of minimal degree (among dominant curves) w.r.t. a given polarization  is minimal. 
All this yields the first equality.
The second equality follows easily from Proposition~\ref{prop:Tcurves}(3).
\end{proof}

\section{Upper bound}\label{sec:up}
Throughout this section, we assume that $\bf i$ is a reduced word and we let $X$ be the Bott-Samelson variety $Z_{\bf i}$. 
We give an upper bound for the Gromov width of $X$ equipped with a K\"ahler form $\omega\in H^2(X,\mathbb R)$ by using Gromov-Witten invariants.  

We thus start by recalling some basics on Gromov-Witten invariants.

Given $A\in H_2(X,\mathbb Z)$, consider the moduli space $\overline{\mathcal M^X_{0,k}}(A)$ of stable maps of genus $0$ into $X$ of class $A$ and with $k$ marked points. 
This space carries a virtual fundamental class $[\overline{\mathcal M]}^{\mathrm{vir}}$ in the rational \v{C}ech homology group $H_d(\overline{\mathcal M_{0,k}}(A),\mathbb Q)$ where $d$ denotes the expected dimension of $\overline{\mathcal{M}^X_{0,k}}(A)$, that is 
$$
d=\mathrm{dim}X+c_1(A)+k-3
$$
where $c_1$ denotes the first Chern class of the tangent bundle $T_X$ of $X$. 

\begin{proposition}\label{prop:smooth}
Let $A$ be the class of a $T$-stable curve of $X$ through the generic point $z_0$ of $X$.
Then the moduli space $\overline{\mathcal M^X_{0,k}}(A)$ is smooth and has the expected dimension.
Moreover, 
$$
[\overline{\mathcal M^X_{0,k}}(A)]=[\overline{\mathcal M]}^{\mathrm{vir}}.
$$
\end{proposition}

\begin{proof} 
For any $T$-stable curve $C$ of $X$ going through $z_0$, we have: $H^1(C, {T_{X}}_{|_C})=0$ by~\cite[Lemma 2.5(i)]{BK21}. Therefore, the moduli space $\overline{\mathcal M^X_{0,k}}(A)$ is unobstructed. The proposition follows; see~Section 2 in~\cite{Lee}.
\end{proof}

Let
$$
ev^k:\overline{\mathcal{M}^X_{0,k}}(A)\longrightarrow X^k
$$
be the evaluation map sending a stable map to the $k$-tuple of its values at the  $k$ marked points.

\begin{corollary}\label{cor:mod-curves}
Let $A$ be the class of a minimal curve of $X$ through the generic point $z_0$ of $X$.
Then the evaluation map $ev^1:\overline{\mathcal{M}^X_{0,1}}(A)\rightarrow X$ is an isomorphism.
\end{corollary}

\begin{proof}
By Proposition~\ref{thm:BKT}, $A=[C_j]$ for some $j$ such that $s_{i_r}\cdots s_{i_{j+1}}(\alpha_{i_j})$ is a simple root.
For such a $j$, we have in particular, that $s_{i_1}...s_{i_{j-1}}s_{i_{j+1}}...s_{i_r}$ is reduced hence
the natural morphism $X\rightarrow Z_{(i_1,....i_{j-1},i_{j+1},...i_r)}$ is a $\mathbb P^1$-fibration with fiber 
over $[\dot s_{i_1},\ldots,\dot s_{i_{j-1}},\dot s_{i_{j+1}},\ldots,\dot s_{i_{}j-1}]$ the curve $C_j$ itself.
It follows that the evaluation map $ev^1$ is bijective. Thanks to Proposition~\ref{prop:smooth}, we can conclude the proof.
\end{proof}

For $\alpha_i\in H^*(X,\mathbb Q)$ with $i= 1,...,k$, the Gromov-Witten invariant is defined to be the rational number
$$
GW^X_{A,k}(\alpha_1,...,\alpha_k):=\int_{[\overline{\mathcal M}]^{\mathrm{vir}}} (ev^k)^*(\alpha_1\times ...\times\alpha_k),
$$
whenever the degrees of $\alpha_1,\ldots,\alpha_k$ sum up to the expected dimension $d$; otherwise it is $0$.

\begin{proposition}\label{prop:GW-non-vanish}
Let $A$ be the class of a minimal curve $C$ of $X$ through $z_0$. 
Then 
$$ 
GW^X_{A,1}(\mathrm{PD}[pt])\neq 0.
$$
\end{proposition}

\begin{proof}
Note that $c_1(A)=2$ by Proposition~\ref{prop:Tcurves}(4).
It follows that $PD[pt]$ satisfies the codimension condition that is, its degree equals the expected dimension which is
$d=\mathrm{dim}X+2-2$. 
By Proposition~\ref{prop:smooth} along with Corollary~\ref{cor:mod-curves}, we have the equality
$$
GW^X_{A,1}(\mathrm{PD}[pt])=\int_{[X]} \mathrm{PD}[pt].
$$
The right hand side being obviously non-equal to $0$, the proposition follows.
\end{proof}

The variety $X$ being projective and smooth and $\omega$ being K\"ahler by assumption, the moduli space $\overline{\mathcal M^X_{0,k}}(A)$ is homeomorphic to the moduli space of stable $J$-holomorphic maps of genus $0$ into $X$ of class $A$ and with $k$-marked points. Here $J$ stands for the complex structure of $X$.
Moreover, the algebraic and symplectic virtual fundamental classes as constructed in~\cite{beh} and \cite{siebert} resp. coincide; see~\cite{siebert}.
As a consequence, Theorem 4.1 in~\cite{HLS} (for peculiar cocycles) reads as follows.

\begin{theorem}[{Gromov}]\label{thm:Gromov}
Let $(X,\omega)$ be K\"ahler and
$A \in H_2(X, \mathbb Z)\setminus \{0\}$ be a second homology class.  
If $GW^X_{A, k}(\mathrm{PD}[pt], \alpha_2, \ldots, \alpha_k)\neq 0$ for some $k$ and $\alpha_i\in H^*(X, \mathbb Q)$ ($i=2,\ldots,k$), then the inequality $w_G(X, \omega)\leq \omega(A)$ holds.
\end{theorem}

 \begin{corollary}\label{cor:upper}
 Let $[\omega]\in H^2(Z_{\bf i},\mathbb R)$ be a K\"ahler form and $A$ be the class of a minimal curve of $Z_{\bf i}$. If the expression ${\bf i}$ is reduced then
$$
w_G(Z_{\bf i}, \omega)\leq \omega(A).
$$  
\end{corollary}
  
  \begin{proof}
  The corollary follows readily from Theorem~\ref{thm:Gromov} and Proposition~\ref{prop:GW-non-vanish} whenever $A=[C_j]$ for some minimal curve $C_j$ of $Z_{\bf i}$. Note that $\omega(A)=\omega([C])$ for every curve $C$ in the minimal family containing $C_j$. This along with Theorem~\ref{thm:BKT} yields the inequality for any minimal curve.
  \end{proof}
  
\section{Lower bound}\label{sec:lower}
In this section, we prove Theorem~\ref{thm:main-lower}: we give a lower bound for the Gromov width.
The bound we obtain is derived from a result of Kaveh's involving Newton-Okounkov bodies.

Newton-Okounkov bodies of projective algebraic varieties are convex bodies 
generalizing momentum polytopes of symplectic toric manifolds; they were introduced about the same time in~\cite{KK12} and \cite{LM}.  
In the literature, one can find several non-equivalent definitions of Newton-Okkounkov bodies for Bott-Samelson varieties. Here, we are concerned with one construction which parallels one of Fujita's definitions.

\subsection{}
Let us thus start by defining the Newton-Okounkov body of interest to us.

Fix a reduced expression  ${\bf i}=(i_1,\ldots,i_r)$ of a given $w\in W$.
Recall the definition of the roots $\beta_i$ given in~\ref{eq:def-beta_i} and that $U_{\beta_i}$ denotes the root subgroup of $G$ corresponding to $\beta_i$.

We shall regard $U_{\beta_{1}}\times \cdots \times U_{\beta_{r}}$ as an affine (open) neighbourhood of the base point 
$z_0$ of $Z_{\bf i}$ via the natural isomorphisms  
\begin{eqnarray}\label{eq:ident-open}
U_{\beta_{1}}\times \cdots \times U_{\beta_{r}}\simeq Bwx_0\simeq Bz_0.
\end{eqnarray}
The first isomorphism is well-known ; it follows mainly from the characterisation of the roots $\beta_i$ recalled in~(\ref{eq:beta_i}) (see e.g. \cite[\S II.13.3]{Jan03} for details). 
The second one is the isomorphism~(\ref{eq:B-orbits}); recall that $x_0$ denotes the base point of $G/B$.

We further identify the function field $\mathbb C(Z_{\bf i})=\mathbb C(U_{\beta_{1}}\times \cdots \times U_{\beta_{r}})$ with the rational function field $\mathbb C(t_1, \ldots, t_r)$ via the isomorphism of algebraic varieties 
$\mathbb C^r\simeq U_{\beta_{1}}\times \cdots \times U_{\beta_{r}}$ given by
$$
(t_1,\ldots,t_r)\longmapsto (\exp(t_1 E_{\beta_1}),\dots, \exp(t_r E_{\beta_r}))
$$
where $E_{\beta_i}$ denotes the root vector associated to the (positive) root $\beta_i$.

The lexicographic order $\leq$ on $\mathbb Z^r$ induces a total order on the set of monomials in the variables $t_1, \ldots, t_r$ (also denoted by $\leq$)
by setting 
$$
t_1^{a_1}\cdots t_r^{a_r} \leq t_1^{a'_1}\cdots t_r^{a'_r}~ \mbox{if and only if} ~(a_1, \ldots, a_r)\leq (a'_1, \ldots, a'_r) ~\mbox{in}~ \mathbb Z^r.
$$ 
Given $f(t_1, \ldots, t_r)=\sum_{j=(j_1, \ldots, j_r)}c_{j}t_1^{j_1}\cdots t_r^{j_r}\in\mathbb C[t_1, \ldots, t_r]$, let 
$(k_1, \ldots, k_r)$ be the maximum tuple among the tuples $(j_1, \ldots, j_r)$ such that $c_j\neq 0$.
Define 
$$
 v_{\beta}(f)=-(k_1, \ldots, k_r).
$$
This induces the following map 
$$
 v_{\beta}: \mathbb C(t_1, \ldots, t_r)\setminus \{0\}\longrightarrow \mathbb Z^r, \quad
 v_{\beta}(\frac{f}{g}):= v_{\beta}(f)- v_{\beta}(g)
\mbox{ for $f, g\in \mathbb C[t_1, \ldots, t_r]\setminus \{0\}$.}
$$

Given an ample line bundle $\mathcal L$ on $Z_{\bf i}$ and a non-zero section 
$\tau \in H^0(Z_{\bf i}, \mathcal L)$,
we shall regard $H^0(Z_{\bf i}, \mathcal L^{\otimes k})$
as a complex vector subspace of the function field $\mathbb C(Z_{\bf i})$ via the map
$$
H^0(Z_{\bf i}, \mathcal L^{\otimes k}) \longrightarrow \mathbb C(Z_{\bf i}), \quad \sigma \longmapsto \frac{\sigma}{\tau^k}.
$$

As in the theory of Newton-Okounkov bodies, let us now consider the following set
$$
S(Z_{\bf i},\mathcal L,  v_{\beta}, \tau):=
\bigcup _{k>0}\{(k,  v_{\beta}(\frac{\sigma}{\tau^k})): \sigma \in H^0(Z_{\bf i}, \mathcal L^{\otimes k})\setminus \{0\}\} \subset \mathbb Z_{\geq 0}\times \mathbb Z^{r}.
$$
Note that $S(Z_{\bf i},\mathcal L,  v_{\beta}, \tau)$ is a semigroup since $v_{\beta}$ is a valuation, which can be easily checked.
Let $\mathcal C(Z_{\bf i},\mathcal L,  v_{\beta}, \tau)$ be the real closed convex cone generated by $S(Z_{\bf i},\mathcal L,  v_{\beta}, \tau)$. 
The Newton-Okounkov body associated to $\mathcal L$, $\tau$ and $ v_{\beta}$  is defined as 
$$
\Delta(Z_{\bf i},\mathcal L, v_{\beta}, \tau)=\left\{x: (1,x)\in \mathcal C(Z_{\bf i},\mathcal L,  v_{\beta}, \tau)\right\}.
$$

\subsection{}\label{sec:4.2}
In this paper, we consider the Newton-Okounkov body associated to a particular section that we now introduce.

Given a dominant weight $\lambda$ (w.r.t. $B$ and $T$), let $V(\lambda)$ denote the simple $G$-module with highest weight $\lambda$ and $v_{\lambda}\in V(\lambda)$ be a highest weight vector. For our purpose,
we consider the morphism associated to the (very) ample line bundle $\mathcal L=\mathcal L_{\bf i,m}$
\begin{equation*}
\begin{array}{llll}
\Psi_{\bf i,m}: & Z_{\bf i} & \longrightarrow & \mathbb P(V(m_1\varpi_{i_1})\otimes\ldots\otimes V(m_r\varpi_{i_r}))\\
& \left[(p_1,\ldots,p_r)\right] & \longmapsto & [p_1 v_{m_1\varpi_{i_1}}\otimes\ldots\otimes p_1p_2\ldots p_r v_{m_r\varpi_{i_r}}].
\end{array}
\end{equation*}

Note that the image of the base point $z_0\in Z_{\bf i}$ through this morphism is $[v_0]$ with
\begin{eqnarray}\label{eq:vector}
v_0=s_{i_1}v_{m_1\varpi_{i_1}}\otimes s_{i_1}s_{i_2}v_{m_2\varpi_{i_2}}\otimes\ldots\otimes wv_{m_r\varpi_{i_r}}.
\end{eqnarray}
Moreover, the morphism $\Psi_{\bf i,m}$ induces 
an isomorphism of $P_{i_1}$-modules
\begin{equation}\label{eq:iso-demazure}
\begin{array}{llll}
\Phi_{\bf i,m}:V_{\bf{i,m}}^*\longrightarrow H^0(Z_{\bf i},\mathcal L_{\bf i,m})
\end{array}
\end{equation}
where $V_{\bf{i,m}}$ denotes the so-called generalized Demazure module. 
As a complex vector space, $V_{\bf{i,m}}$ is generated by the following vectors
$$
F_{i_1}^{a_1}(v_{m_1\varpi_{i_1}}\otimes F_{i_2}^{a_2}(v_{m_2\varpi_{i_2}}\otimes\ldots\otimes F_{i_{r-1}}^{a_{r-1}}(v_{m_{r-1}\varpi_{i_{r-1}}}\otimes F_{i_r}^{a_r}v_{m_r\varpi_{i_r}})\ldots))
$$
with $a_j\in\mathbb N$ and $F_{i_j}$ being the root vector associated to the root $-\alpha_{i_j}$;
see Theorem 6 and Section~4.2 in~\cite{LLM}.

We set
$$
f_0=s_{i_1}v_{m_1\varpi_{i_1}}^*\otimes s_{i_1}s_{i_2}v_{m_2\varpi_{i_2}}^*\otimes\ldots\otimes wv_{m_r\varpi_{i_r}}^*\quad\mbox{ and }\quad
\varphi_0=\Phi_{\bf i,m}(f_0).
$$

\begin{lemma}\label{lem:varphi_0}
The section $\varphi_0\in H^0(Z_{\bf i}, \mathcal{L}_{\bf i,m})$ does not vanish on the open subset $Bz_0$ of $Z_{\bf i}$.
\end{lemma}

\begin{proof}
Note that $f_0$ does not vanish on $U_{\beta_1}\times\ldots\times U_{\beta_r}[v_0]$, by a simple consideration on weights.
Thanks to~(\ref{eq:iso-demazure}) and the isomorphism $Bz_0\simeq U_{\beta_{1}}\times \cdots \times U_{\beta_{r}}$ recalled in~(\ref{eq:ident-open}), the result follows.
\end{proof}

Lemma~\ref{lem:varphi_0} enables us to introduce the Newton-Okounkov body associated to $\varphi_0$, that is
\begin{eqnarray}
\Delta_{\bf i,m}:=\Delta(Z_{\bf i},\mathcal L_{\bf i,m}, v_{\beta}, \varphi_0).
\end{eqnarray}

\begin{remark}\label{rem:Fujita}\
\begin{enumerate}
\item 
By considering the open embedding $s_{i_1}U^-_{\alpha_{i_1}}\times\ldots\times s_{i_r} U^-_{\alpha_{i_r}}\hookrightarrow Z_{\bf i}$ 
given by $(s_{i_1}u_1,\ldots, s_{i_r}u_r)\mapsto [(s_{i_1}u_1,\ldots, s_{i_r} u_r)]$,
one naturally identifies $\mathbb C(Z_{\bf i})$ with $\mathbb C(t'_1,\ldots, t'_r)$.
In~\cite[\S 8]{Fujita}, Fujita introduces the valuation $v'_{\bf i}$ on  $\mathbb C(t'_1,\ldots, t'_r)$ defined up to the lexicographic order on $\mathbb Z^r$ as above and studies the Newton-Okounkov body 
$\Delta(Z_{\bf i},\mathcal L,v'_{\bf i},\varphi_0)$.
\item
By noticing that $\exp(a_1 E_{\beta_1})\cdots\exp(a_r E_{\beta_r})wx_0=
s_{i_1}\exp(a_1 F_{i_1})\cdots s_{i_r}\exp(a_r F_{i_r})x_0$ 
for any $(a_1,\ldots,a_r)\in\mathbb C^r$, one easily sees that the semigroups 
$S(Z_{\bf i},\mathcal L,v'_{\bf i},\varphi_0)$ and $S(Z_{\bf i},\mathcal L_{\bf i,\bf m},  v_{\beta}, \varphi_0)$ coincide and so do
the convex bodies  $\Delta(Z_{\bf i},\mathcal L,v'_{\bf i},\varphi_0)$ and $\Delta_{\bf i,m}$.
\end{enumerate}
\end{remark}

Thanks to Remark~\ref{rem:Fujita}(2), \cite[Cor. 8.3]{Fujita} reads as follows.

\begin{theorem}[{Fujita}]\label{thm:Fujita}\
\begin{enumerate}
    \item 
 The semigroup $S(Z_{\bf i},\mathcal L_{\bf i,\bf m},  v_{\beta}, \varphi_0)$  is finitely generated. 
\item The Newton-Okounkov body $\Delta_{\bf i,m}$ is a convex polytope.
\end{enumerate}
\end{theorem}

\subsection{}
We are now ready to state Kaveh's result which is a consequence of the following theorem. 

\begin{theorem}[{\cite[Section 4.2]{LMS13}}]\label{thm:embed-toric}
Let $(X,\omega)$ be a proper connected symplectic toric $2n$-dimensional manifold equipped with a momentum map.
Suppose there exists a $n$-dimensional simplex of size $\kappa$ contained in the momentum image.
Then the Gromov width of $(X,\omega)$ is at least $\kappa$.
\end{theorem}

Here a simplex in $\mathbb R^m$ is said to be of size $\kappa$ if it can be obtained
from the simplex $\{(x_i)_i\in\mathbb R_{>0}^m: x_1+\ldots+ x_m<\kappa\}$ by a linear transformation in $\mathrm{GL}(m, \mathbb Z)$ and a translation of $\mathbb R^m$.

Corollary 11.4 in~\cite{Kav} specialised to the case of Bott-Samelson varieties and the Newton-Okounkov bodies 
$\Delta_{\bf i,m}$ reads as follows.

\begin{corollary}[{Kaveh}]\label{cor:Kaveh}
Let $(Z_{\bf i},\mathcal L_{\bf i,m})$ be a Bott-Samelson variety equipped with a very ample line bundle.
Then the Gromov width of $(Z_{\bf i}, \mathcal L_{\bf i,m})$ is bounded from below by the supremum of the size of (open) simplices that fit in the interior of Newton-Okounkov body $\Delta_{\bf i,m}$.
\end{corollary}

\begin{proof}
As a sake of convenience, we outline Kaveh's proof in our particular setting.

Since the monoid $S(Z_{\bf i,\bf m},\mathcal L_{\bf i,\bf m},  v_{\beta}, \varphi_0)$ is finitely generated (Theorem~\ref{thm:Fujita}), 
the variety $Z_{\bf i}$ admits a flat degeneration to the projective (non necessarily normal) toric variety 
$X_0=\mathrm{Proj}\,\mathbb C[S(Z_{\bf i},\mathcal L_{\bf i,\bf m},  v_{\beta}, \varphi_0)]$ thanks to~\cite[Theorem 1]{And13}. 
The normalisation of 
$X_0$ is the projective (normal) toric associated to the polytope $\Delta_{\bf i,m}$.
Moreover, by Theorem A along with Theorem B in~\cite{HK15}, there exists a K\"ahler form $\omega_0$ on the smooth locus $U_0$ of $X_0$ such that
\begin{enumerate}
    \item 
 $(U_0,\omega_0)$ is symplectomorphic to $(U,\omega_{\mathcal L_{\bf i,m}})$ for some open subset $U$ of $Z_{\bf i}$ and
\item
the momentum image of the symplectic toric manifold $(U_0,\omega_0)$ contains the interior of the Newton-Okounkov body
$\Delta_{\bf i,m}$.
\end{enumerate}
The corollary thus follows from Theorem~\ref{thm:embed-toric}.
\end{proof}
\begin{remark}
Kaveh's result does not require that the Newton-Okounkov body be a convex polytope, but the proof becomes longer.
\end{remark}

\subsection{}
We now proceed to the proof of Theorem~\ref{thm:main-lower}: we shall exhibit a simplex of the advertized size in the Newton-Okounkov body $\Delta_{\bf i,m}$ and apply Corollary~\ref{cor:Kaveh}.

Given ${\bf m}\in \mathbb Z_{>0}^r$, for each $1\leq j\leq r$, we set 
$$
\lambda_j=m_j\omega_{i_j}\quad\mbox{ and }\quad
\ell_j=m_j+\sum_{j<\ell\leq r} \langle \lambda_{\ell}, s_{i_\ell}\cdots s_{i_{j+1}}(\alpha_{i_{j}}^{\vee})\rangle.
$$
\begin{lemma}\label{lem:aux+ell}
The roots $s_{i_{k+1}}\ldots s_{i_{j-1}}(\alpha_{i_{j}})$
and $s_{i_{\ell}}\ldots s_{i_{j+1}}(\alpha_{i_{j}})$ are positive for every $1\leq k<j\leq r$ and every $j<\ell\leq r$ respectively.
\end{lemma}

\begin{proof}
We show the assertion for the roots $s_{i_{k+1}}\ldots s_{i_{j-1}}(\alpha_{i_{j}})$, the 
proof being similar for the other roots.
Let us fix $j$ and proceed by induction on $k$.
We thus start by showing that $s_{i_2}\ldots s_{i_{j-1}}(\alpha_{i_{j}})$ is a positive root for all $j\geq 2$, the case $k=2$.
Because $s_{i_1}s_{i_2}\ldots s_{i_{j-1}}(\alpha_{i_{j}})=\beta_j$ is a positive root, if $s_{i_2}\ldots s_{i_{j-1}}(\alpha_{i_j})$ were a negative root, it 
would be equal to $-\alpha_{i_1}$.
This would imply that $\beta_j=\beta_1$ -- a contradiction since the roots $\beta_j$ are pairwise distinct and $j\neq 1$.
Next,  we consider $s_{i_{\ell}}\ldots s_{i_{j-1}}(\alpha_{i_{j}})$ with $2\leq \ell\leq j-1<r$.
By induction hypothesis, $s_{i_{\ell-1}}\ldots s_{i_{j-1}}(\alpha_{i_{j}})$ is a positive root hence by arguing similarly as for the case $k=2$, we get that
$s_{i_{\ell}}\ldots s_{i_{j-1}}(\alpha_{i_{j}})$ is also a positive root; otherwise it would be equal to $-\alpha_{i_{\ell-1}}$ and in turn we would have:
$\beta_j=s_{i_1}\ldots s_{i_\ell}\ldots s_{i_{j-1}}(\alpha_{i_j})=s_{i_1}\ldots s_{i_{\ell-2}}(\alpha_{i_{\ell-1}})=\beta_{\ell-1}$ -- a contradiction.
\end{proof}

Recall the definition of the linear form $f_0\in V(\lambda_1)^*\otimes\ldots\otimes V(\lambda_r)^*$ as well as the morphism $\Phi_{\bf i,m}$ introduced in Subsection~\ref{sec:4.2}.
Set
$$
f_j=F_{\beta_j}^{\ell_j}f_0
\quad\mbox{ and }\quad
\varphi_j=\Phi_{\bf i}(f_j)
$$
with $F_{\beta_j}=E_{-\beta_j}$ being the root operator associated to the negative root $-\beta_j$. Note that $\ell_j$ is positive by Lemma~\ref{lem:aux+ell}.

\begin{lemma}\label{lem:vanishing}
For each $j$, we have 
$E_{\beta_j}(s_{i_1}\ldots s_{i_k}v_{\lambda_k})=0$ for every $1\leq k<j\leq r$.
In particular, we have the equality
\begin{equation}\label{eqn:lemma-vanish-dual}
f_j=s_{i_1}v_{\lambda_1}^*\otimes\ldots\otimes F_{\beta_j}^{\ell_j}\left(s_{i_1}\ldots s_{i_j}v_{\lambda_j}^*\otimes\ldots\otimes wv_{\lambda_r}^*\right).
\end{equation}
\end{lemma}

\begin{proof}
Lemma~\ref{lem:aux+ell} implies that 
$
\langle s_{i_1}\ldots s_{i_k}\lambda_k,\beta_j^\vee\rangle=\langle \lambda_k,s_{i_{k+1}}\ldots s_{i_{j-1}}(\alpha_{i_{j}})^\vee\rangle\geq 0
$
and, in turn
$E_{\beta_j}(s_{i_1}\cdots s_{i_k}v_{\lambda_k}) =0$ for all $j\neq 1$ because $s_{i_1}\cdots s_{i_k}v_{\lambda_k}$ is an extremal weight vector. By duality, this proves the lemma.
\end{proof}

Recall the definition of the weight vector $v_0$ given in~(\ref{eq:vector}).

\begin{lemma}\label{lem:E-vanish}
The equality 
$E^{\ell_j+1}_{\beta_j}(v_0)=0$ holds for every $1\leq j\leq r$.
Moreover, $E^{\ell_j}_{\beta_j}(v_0)\neq 0$.
\end{lemma}

\begin{proof}
Note that we have 
$$
a_{j\ell}:=\langle s_{i_1}\ldots s_{i_\ell}\lambda_\ell,\beta_j^\vee\rangle=\begin{cases}-m_j
& \text{if}~j=\ell \\
 -\langle \lambda_{\ell}, s_{i_\ell}\cdots s_{i_{j+1}}(\alpha_{i_{j}}^{\vee})\rangle & \text{if}~ j<\ell\leq r
\end{cases}.
$$
By Lemma~\ref{lem:aux+ell}, the integers $a_{j\ell}$ are negative and since 
$s_{i_1}\cdots s_{i_\ell}v_{\lambda_\ell}$ is an extremal weight vector, 
$E_{\beta_j}^{1-a_{j\ell}}(s_{i_1}\cdots s_{i_\ell}v_{\lambda_\ell}) =0$ for all $j\leq \ell\leq r$.
Moreover, by definition of $\ell_j$, the integers $-a_{j\ell}$, for $\ell=j,\ldots,r$, sum up to $\ell_j$. 
All this implies the equality 
$E^{\ell_j+1}_{\beta_j}(s_{i_1}\ldots s_{i_j}v_{\lambda_j}\otimes\ldots\otimes wv_{\lambda_r})=0$.
We conclude the proof by applying Lemma~\ref{lem:vanishing}.
\end{proof}

\begin{lemma}\label{lem:non trial section}
The sections $\varphi_j\in H^0(Z_{\bf i},\mathcal L_{\bf i,m})$ are not trivial.
\end{lemma}

\begin{proof}
Let $u_j=\exp(E_{\beta_j})\in U_{\beta_j}$. 
By Lemma~\ref{lem:E-vanish} together with Lemma~\ref{lem:vanishing}, we have
\begin{eqnarray*}
u_j (v_0) & = & E_{\beta_j}^0 (v_0)  +  \ldots + \frac{E_{\beta_j}^{k}}{k!}(v_0)+\ldots+ \frac{E_{\beta_j}^{\ell_j}}{\ell_j!}(v_0)\\
                  & = & v_0                  +  \ldots +  s_{i_1}v_{\lambda_1}\otimes\ldots\otimes s_{i_1}\ldots s_{i_{j-1}}v_{\lambda_{j-1}}\otimes \frac{E_{\beta_j}^{k}}{k!}\left(s_{i_1}\ldots s_{i_j}v_{\lambda_j}\otimes\ldots\otimes wv_{\lambda_r}\right)+\ldots.
\end{eqnarray*}
Besides, Equality~(\ref{eqn:lemma-vanish-dual}) implies
$f_j(E_{\beta_j}^k(v_0))=0$ for all $k<\ell_j$.
Moreover, $f_j(E_{\beta_j}^{\ell_j}(v_0))\neq 0$ since $E_{\beta_j}^{\ell_j}(v_0)\neq 0$ (Lemma~\ref{lem:E-vanish}).
Therefore, we have $f_j(u_j(v_0))\neq 0$ and the lemma follows.
\end{proof}

\begin{proposition}\label{prop:poly-fct}
For each $1\leq j\leq r$, we have
$\varphi_j/\varphi_0=a_j t_j^{\ell_j}\in\mathbb C[t_1,\ldots, t_r]$ for some $a_j\in\mathbb C$.
\end{proposition}

\begin{proof}
Take a section $\varphi\in H^0(Z_{\bf i},\mathcal L_{\bf i,m})$.
Thanks to Lemma~\ref{lem:varphi_0}, the quotient $\varphi/\varphi_0$ is a regular function on $Bz_0$ and in turn can be regarded as an element of $\mathbb C[t_1,\ldots,t_r]$.

By arguing as in the proof of Lemma~\ref{lem:non trial section} and using the fact that the roots $\beta_j$ are pairwise distinct, we show that $f_j(U_{\beta_k}v_0)=0$ for all $k\neq j$.
It follows that $\varphi_j/\varphi_0\in\mathbb C[t_j]$.
To conclude the proof, we observe that in the course of the proof of Lemma \ref{lem:non trial section}, we got that
$f_j(u_{\beta_j}v_0)=f_j((aE_{\beta_j})^{\ell_j}v_0/\ell_j !)$ with $u_{\beta_j}=\exp(a E_{\beta_j})$ and $a\in\mathbb C$.
\end{proof}

\begin{corollary}\label{cor:lower}
The polytope $\Delta_{\bf i, m}$ contains a simplex of size 
$
\kappa=\min\{ \ell_j: 1\leq j\leq r\}.
$
\end{corollary}

\begin{proof} 
First, note that the polytope $\Delta_{\bf i, m}$ contains the origin since $v_{\bf \beta}(\varphi_0)=0$.

From the definition of the valuation $v_{\beta}$ and Proposition~\ref{prop:poly-fct}, we get
$$
v_{\bf \beta}(\varphi_j)=-(0,\ldots,0,\ell_j,0,\ldots,0)=-\ell_j e_j.
$$
This together with the convexity of $\Delta_{\bf i,m}$ (Theorem~\ref{thm:Fujita}) implies the corollary.
\end{proof}

\begin{proof}[Proof of Theorem~\ref{thm:main-lower}]
We first consider the case of an integral K\"ahler form $\omega$ of $Z_{\bf i}$, that is 
$\omega$ is the pullback of the Fubini-Study form on the projectivization of $H^0(Z_{\bf i},\mathcal L)$ for some very ample line bundle $\mathcal L$ of $Z_{\bf i}$.
By Theorem~\ref{thm:LT}, the line bundle $\mathcal L$ equals $\mathcal L_{\bf i, m}$ for some 
${\bf m}\in \mathbb Z_{>0}^r$. We thus write $\omega=\omega_{\bf m}$.

Corollary~\ref{cor:Kaveh} and Corollary~\ref{cor:lower} give the inequality:
\begin{equation}\label{ineq:integral}
w_G(Z_{\bf i}, \omega_{\bf m})\geq \min\big\{ \ell_j: 1\leq j\leq r\big\}.
\end{equation}

We next consider any $2$-form $\omega_{\bf m'}$ of $Z_{\bf i}$ with ${\bf m'}\in\mathbb Q_{>0}^r$, namely  $\omega_{\bf m'}$ is the $2$-form associated to $\mathcal L_{\bf m'}\in\mathrm{Pic}(Z_{\bf i})\otimes_\mathbb Z \mathbb Q$.
Let $a\in \mathbb Z_{>0}$ be such that $a\omega_{\bf m'}$ is an integral K\"ahler form, i.e. $a\omega_{\bf m'}=\omega_{\bf m}$ with ${\bf m}=a{\bf m'}\in\mathbb Z_{>0}^r$.
Since $w_G(Z_{\bf i}, a\omega_{\bf m'})=aw_G(Z_{\bf i}, \omega_{\bf m'})$, we shall derive Inequality~(\ref{ineq:geo}) from Inequality~(\ref{ineq:integral}).

Recall the definition of the $T$-stable curve $C_j$ from Section~\ref{sec:BS-varieties}.
Thanks to Corollary~\ref{cor:dominant-minimal}, Inequality~(\ref{ineq:integral}) reads as
\begin{equation*}
w_G(Z_{\bf i}, \omega_{\bf m})\geq 
\min\big\{ \omega_{\bf m}([C_j]): 1\leq j\leq r
\}=\min\{ \omega_{\bf m}([C]): \mbox{$C$ minimal}\}.
\end{equation*}
This ends the proof of Theorem~\ref{thm:main-lower}.
\end{proof}

\begin{proof}[Proof of Corollary~\ref{cor:coadorbit}]
Arguing as in the proof of Theorem~\ref{thm:main-lower}, we can assume that $\lambda$ is integral.
Under this assumption, the coadjoint orbit $O_\lambda$ equipped with its $K$-invariant complex structure is thus a flag variety and, in turn, a Schubert variety $X(w)$ associated to the longest element $w$ of the Weyl group of a parabolic subgroup $P_\lambda$ of $G$ containing the Borel subgroup $B$.

Let $\bf{i}$ be any reduced expression of $w$ and let $\mathcal L_\lambda$ denote the ample line bundle on $O_\lambda$ associated to $\lambda$.
The pull-back via the morphism $\pi_r: Z_{\bf{i}}\rightarrow O_\lambda$ of $\mathcal L_\lambda$ being generated by its global sections, it is equal to $\mathcal L_{\bf{i,m}}$ for some $m\in \mathbb Z_{\geq 0}^r$ by Theorem~\ref{thm:LT}. Note that the construction of the Newton-Okounkov body
$\Delta_{\bf{i,m}}=\Delta(Z_{\bf{i}},\mathcal L_{\bf{i,m}},v_\beta,\tau)$ can also be performed for $\mathcal L_{\bf{i,m}}$ generated by its global sections and non necessarily ample.
Thanks to the birationality of the morphism $\pi_r$, we can regard the Newton-Okounkov body 
$\Delta(O_\lambda,\mathcal L_\lambda, v_\beta,\tau_\lambda)$ as the Newton-Okounkov body $\Delta_{\bf{i,m}}$
with $\tau_\lambda$ being the lowest weight vector of the dual of $V(\lambda)$ and
such that $\tau_\lambda(v_\lambda)=1$.

Let $y_0$ be the base point of $G/P_\lambda$. 
Recall that the $T$-stable curves through $wy_0$ are the $U_\alpha$-orbit closures of $wy_0$ 
within $O_\lambda=G/P_\lambda$ with $\alpha$ being a positive root non-orthogonal to $w(\lambda)$. Note that the set of these roots $\alpha$ coincides with the set of roots $\beta_j$ (defined in (\ref{eq:beta_i})).
Moreover, for any such curve, say $C_\alpha$, the equalities  $(w(\lambda),-\alpha^\vee)=\mathcal L_\lambda\cdot C_\alpha=\omega_{\bf{m}}(C_j)$ hold (see e.g.~\cite{BK21}). This enables to conclude the proof.
\end{proof}

\section{Gromov widths of Bott manifolds and of Bott-Samelson varieties}~\label{sec:toric-BS}

The main goal of this section is to derive the Gromov widths of Bott-Samelson varieties which can be degenerated into Bott manifolds from the Gromov widths of the latter manifolds (computed in~\cite{HLS}).
We would like also to draw the reader's attention on Proposition~\ref{prop:bott-case-reformul} which gives a reformulation of the combinatorial expression of the Gromov width obtained in loc. cit. in terms of the symplectic areas of minimal curves. This result is independent from the rest of the paper.

\subsection{}
We start by reviewing a few basic facts on generalized Bott manifolds.

An $m$-stage {\it generalized Bott tower} is a sequence of complex projective space bundles
\begin{equation*}
B_m\stackrel{\pi_m}{\longrightarrow}
B_{m-1}\stackrel{\pi_{m-1}}{\longrightarrow}\cdots\stackrel{\pi_2}{\longrightarrow}
B_1\stackrel{\pi_1}{\longrightarrow}B_0=\{pt\},
\end{equation*}
where $B_j=\mathbb{P}(\mathcal{L}_j^{(1)} \oplus
\cdots \oplus \mathcal{L}_j^{(n_j)} \oplus \mathcal{O}_{B_{j-1}})$
for some line bundles $\mathcal{L}_j^{(1)}, \ldots, \mathcal{L}_j^{(n_j)}$ over $B_{j-1}$.

Since the Picard group of $B_{j-1}$ is isomorphic to $\mathbb{Z}^{j-1}$ for any $j=1, \ldots, m$,
each line bundle $\mathcal{L}_j^{(k)}$
corresponds to a $(j-1)$-tuple of integers
$(a_{j, 1}^{(k)}, \ldots, a_{j, j-1}^{(k)}) \in \mathbb{Z}^{j-1}$.
The variety $B_m$ is thus determined by a collection of integers
\begin{equation*}
a_{j, l}^{(k)} \quad\mbox{with}~ 2 \leq j \leq m, 1 \leq k \leq n_j ~\mbox{and}~ 1 \leq l \leq j-1.
\end{equation*}
$B_m$ is called the $m$-stage generalized Bott manifold associated to the collection $(a_{j, l}^{(k)})$.

Recall that any projective bundle of sum of line bundles over a smooth toric variety is again a smooth toric variety (see \cite[Proposition 7.3.3]{cox}). Therefore $B_m$ is a smooth projective toric variety. 

We now describe the fan of the $m$-stage generalized Bott manifold $B_m$ associated to the collection $(a_{j, l}^{(k)})$.
Let $n=n_1+\cdots+n_m$ and let $\{e_1^1, \ldots, e_1^{n_1}, \ldots, e_m^1, \ldots, e_m^{n_m}\}$
be the standard basis for $\mathbb{Z}^n$.
For $l=1, \ldots, m$, we define
\begin{equation*}
 u_l^k=e_l^k \hspace{0.5cm} \mbox{for} \hspace{0.5cm} k=1, \ldots, n_l \quad \mbox{and}
 \end{equation*}
  \begin{equation*}
u_l^0=-\sum_{k=1}^{n_l}e_l^k+\sum_{j=l+1}^m\sum_{k=1}^{n_j}a_{j, l}^{(k)}e_j^k.
\end{equation*}
Note that $u_1^0, \ldots, u_1^{n_1}, \ldots, u_m^0, \ldots, u_m^{n_m} (\in \mathbb Z^n)$ are ray generators.

Given ${\bf k}=(k_1,\ldots,k_m)$ with $0\leq k_l\leq n_l$ for $1\leq l\leq m$, let
$$
\mathcal C_{{\bf k}}=\mathrm{Cone}\big( u_l^k: 1\leq l\leq m, 0\leq k\leq n_l,\, k\neq k_l\big)
$$
be the $n$-dimensional cone in $\mathbb R^n$ generated by all $u_l^k$ but the $u_l^{k_l}$'s.

\begin{proposition}\label{prop:fan}
Let $\Sigma$ be the fan in $\mathbb R^n$ whose maximal cones consist of the cones $\mathcal C_{\bf k}$, ${\bf k}\in [0,n_1]\times\ldots\times[0,n_m]$. 
Then $\Sigma$ is a smooth complete fan in $\mathbb{R}^n$.

Furthermore, the generalized Bott manifold defined by the collection $(a_{j, l}^{(k)})$ is the toric variety associated to the fan $\Sigma$. 
\end{proposition}

Since $B_m$ is a toric manifold, $H_2(B_m,\mathbb R)$ is isomorphic to the $\mathbb R$-span of the torus stable prime divisors $D_l^k$ of $B_m$, the latter corresponding to the ray generators $u_l^k$ of the fan $\Sigma$.
Given $[\omega]\in H_2(B_m,\mathbb R)$, by abuse of notation, we thus write
\begin{equation}\label{eq:divisor}
[\omega]=\sum_{l,k} \lambda_l^k [D_l^k]\quad\mbox{ with $\lambda_l^k\in\mathbb R$, $1\leq l\leq m$ and  $0\leq k\leq n_l$}.
\end{equation}

For $1\leq l \leq m$, we set
$$
u(l)= \sum_{k=0}^{n_{l}} u_{l}^k \quad \text{and} \quad \lambda(l)= \sum_{k=0}^{n_{l}} \lambda_{l}^k.
$$ 

\begin{theorem}[{\cite[Theorem 1.1]{HLS}}]\label{thm:HLD}
Let $(B_m,\omega)$ be the $m$-stage generalized Bott manifold associated to the collection $(a_{j, l}^{(k)})$ and equipped with a symplectic $2$-form $\omega$ given as in~(\ref{eq:divisor}). Then
$$
w_G(B_m, \omega)=\min\{ \lambda(l): u(l)=0\}.
$$
\end{theorem}

\subsection{} 
Let us recall the combinatorial description of the minimal rational curves on complete toric varieties obtained in~\cite{CFH}. 

Let $X$ be any smooth complete toric variety and $\Sigma$ be its fan.
By $\Sigma(1)$ we denote the set of all primitive generators of one-dimensional cones in the fan $\Sigma$.   

\begin{definition}[{\cite{Bat}}]\label{def:primcoll}
A non-empty subset $\mathfrak P=\{x_1, \ldots, x_k\}$ of $\Sigma(1)$ is called a 
\emph{primitive collection} if, 
for any $1\leq i\leq k$, the set $\mathfrak P\setminus \{x_i\}$ generates a $(k-1)-$dimensional cone in $\Sigma$, while $\mathfrak P$ does not generate a $k$-dimensional cone in $\Sigma$.  
\end{definition}

For a primitive collection $\mathfrak P=\{x_1, \ldots, x_k\}$  of $\Sigma(1)$, let $\sigma(\mathfrak P)$ be the unique cone in $\Sigma$ that contains $x_1+\cdots +x_k$ in its interior. 
Let $y_1, \ldots, y_m$ be generators of $\sigma(\mathfrak P)$. 
There thus exists a unique equation  such that 
$$
x_1+\cdots+x_k=a_1y_1+\ldots+a_my_m \quad\mbox{ with }\quad a_i\in \mathbb Z_{>0}.
$$
The equation $x_1+\cdots+x_k-a_1y_1-\ldots-a_my_m=0$ is called the \emph{primitive relation} of $\mathfrak P$. 
The \emph{degree} of $\mathfrak P$ is 
$$
deg(\mathfrak P)=k-\sum_{i=1}^{m}a_i.
$$

\begin{theorem}[{\cite[Proposition 3.2 and Corollary 3.3]{CFH}}]\label{Fu} 
Let $X$ be a smooth complete toric variety.
\begin{enumerate}
      \item There is a bijection between minimal rational components of degree $k$ on $X$ and primitive collections $\mathfrak P=\{x_1, \ldots, x_k\}$ of $\Sigma(1)$ such that $x_1+\cdots+x_k=0$.
      \item There exists a minimal rational component in $\mathrm{RatCurves}(X)$.
\end{enumerate}
\end{theorem} 

For later use, we recall briefly the idea of the proof of this theorem. 
\begin{proof}[The idea of the proof of Theorem \ref{Fu}]
For a given family $\mathcal K$ of minimal rational curves on $X$ of degree $k$, there exists a  torus invariant open subset $U\subset X$ such that $U\simeq (\mathbb C^*)^{n+1-k}\times \mathbb P^{k-1}$ (see \cite[Corollary 2.6]{CFH}) such that
the lines in the factor $\mathbb P^{k-1}$ give general members of $\mathcal K$. 
The fan defining $U$ is the fan of $\mathbb P^{k-1}$ viewed as a fan in $\mathbb R^n$. This gives a collection $\{x_1, \ldots, x_k\}$ of $\Sigma(1)$ which is primitive and such that $x_1+\cdots+x_k=0$. 

For the converse, assume that  $\mathfrak P=\{x_1, \ldots, x_k\}$ is a primitive collection of $X$ such that $x_1+\cdots+x_k=0$. Now consider the subfan $\Sigma'$ of $\Sigma$ given by the collection of cones in $\Sigma$ generated by the subsets of $\mathfrak P$. Then the toric variety $U_{\Sigma'}$ associated with $\Sigma'$ is an open subset of $X$ and   $U_{\Sigma'}\simeq (\mathbb C^*)^{n+1-k}\times \mathbb P^{k-1}$. Let $C_k$ be a line in the factor $\mathbb P^{k-1}$. Then the deformations of $C_k$ give a minimal family of rational curves in $X$ of degree $k$. 
\end{proof}

\subsection{}
For $1\leq l\leq m$, define $\mathfrak P_l:=\{u_l^0, u_l^1, \ldots, u_l^{n_l}\}$.

\begin{lemma}\label{Pc}
The set of all primitive collections of the generalized Bott manifold $B_m$ is 
$\{\mathfrak P_l: 1\leq l\leq m\}$.
\end{lemma}

\begin{proof}
This follows readily from the description of the fan of $B_m$
(Proposition~\ref{prop:fan}) together with the definition of primitive collections (Definition~\ref{def:primcoll}).
\end{proof}

Here is the advertised reformulation of Theorem~\ref{thm:HLD} in terms of minimal rational curves.

\begin{proposition}\label{prop:bott-case-reformul}
Keep the notation as in Theorem \ref{thm:HLD}.
 $$\min\{\mathcal L\cdot C: C\subset X~\mbox{is a minimal rational curve}~\}= \min\{ \lambda(l): u(l)=0\}.$$
\end{proposition}

\begin{proof} Recall the definition of the fan  $\Sigma$ of $X$; see Proposition \ref{prop:fan}. 
By Theorem \ref{Fu}, any family $\mathcal K_l$ of minimal rational curves corresponds to a primitive collection $\mathfrak P_l$ with $u(l)=0$. 
Given such a primitive collection $\mathfrak P_l$, consider the subfan  $\Sigma'$ of $\Sigma$ given by the collection of cones in $\Sigma$ generated by the subsets of $\mathfrak P_l$. 
Then the fan $\Sigma'$ is isomorphic to the fan of $\mathbb P^{n_l}$ (see the proof of Theorem~\ref{Fu}).
To any primitive relation $u(l)=0$ viewed in the fan of $\mathbb P^{n_l}$, we can associate two maximal cones in $\Sigma'$, say $\sigma$ and $\sigma'$, such that the intersection $\tau= \sigma\cap\sigma'$ is a cone of codimension $1$ in $\Sigma'$. 
Let $C_l$ be the torus invariant curve in $\mathbb P^{n_l}$ associated to $\tau$. Note that $C_l$ is isomorphic to the projective line $\mathbb P^1$ (see \cite[Section 6.3, p.289]{cox}).
Then the family $\mathcal K_l$ associated to $\mathfrak P_l$ is obtained by deformation of the curve  $C_l$ (see the proof of Theorem~\ref{Fu}). 

Besides, the relation $u(l)=0$ corresponds to the element $R(l)=(b_{\rho})_\rho\in N_1(X)\subset \mathbb R^{|\Sigma(1)|}$, the group of numerical classes of 1-cycles on $X$,  with $$b_{\rho}=\begin{cases}
1 & \mbox{if $\rho \in \mathfrak P_l$}\\
0 & \mbox{otherwise.}
\end{cases}
$$  
By~\cite[Proposition 6.4.1]{cox}, we see that 
the intersection number $\mathcal L\cdot R(l)$ equals $\lambda(l)$.
Finally, since $\mathcal L\cdot C =\mathcal L\cdot C'$ for any $C,C'\in\mathcal K_l$,  the proof follows.
\end{proof}

\begin{corollary}\label{cor:reform}
$w_G(X, \mathcal L)=\min\{\mathcal L\cdot C: C\subset X~\mbox{is a minimal rational curve}~\}.$
\end{corollary}

\begin{proof}
This follows readily from Theorem~\ref{thm:HLD} and Proposition~\ref{prop:bott-case-reformul}.
\end{proof}


\subsection{}
For some appropriate choice of $(\bf i,m)$, Bott-Samelson varieties $Z_{\bf i}$ equipped with an ample line bundle $\mathcal L_{\bf i,m}$ can be degenerated into Bott manifolds using the theory of Newton-Okounkov bodies. 
These toric degenerations were derived in~\cite{HY} from some previous constructions of Grossberg and Karshon (see~\cite{GK}  and~\cite{Pas} also). Next, we recall the main properties of these toric degenerations.

To define properly the relevant Newton-Okounkov bodies, denoted below by $P_{\bf i,m}$, we need the following functions.
Given $(\bf{i,m})$, we set
\begin{equation}
\begin{array}{lll}
A_r &= & m_r,\\
A_{r-1}(x_r)& = &\langle m_{r-1}\varpi_{i_{r-1}}+m_r\varpi_{i_r}-x_r\alpha_{i_r},\alpha_{i_{r-1}}^{\vee}\rangle,\\
A_{r-2}(x_{r-1},x_r)& = &\langle m_{r-2}\varpi_{i_{r-2}}+m_{r-1}\varpi_{i_{r-1}}+ m_r\varpi_{i_r}-x_{r-1}\alpha_{i_{r-1}}-x_r\alpha_{i_r},\alpha_{i_{r-2}}^{\vee}\rangle,\\
 \quad \vdots\\
A_1(x_2,\dots,x_r)& = &\langle m_1\varpi_{i_1}+m_2\varpi_{i_2}+\ldots + m_r\varpi_{i_r}-x_2\alpha_{i_2}-\cdots-x_r\alpha_{i_r},\alpha_{i_1}^{\vee}\rangle.
\end{array}
\end{equation}

The polytope $P_{\bf i,m}$ consists of the points $(x_1,\ldots, x_r)\in\mathbb R^r$ satisfying the inequalities
$$
0\leq x_j\leq A(x_{j+1},...,x_n),\quad 1\leq j\leq r.
$$

We now recall the definition of the technical assumption on the pair $(\bf i,m)$ needed for the construction of the toric degenerations.

\begin{definition}
We say that \emph{the pair $(\bf i,m)$ satisfies the condition (P-k)} when the following holds:
If $(x_{k+1}, \ldots, x_r)$ satisfies the inequalities
\begin{equation*}
\begin{array}{ccccl}
0 & \leq & x_r     & \leq & A_r \\
0 & \leq & x_{r-1} & \leq & A_{r-1}(x_r) \\
  &      & \vdots  && \\
0 & \leq & x_{k+1} & \leq & A_{k+1}(x_{k+2}, \ldots, x_r)\\
\end{array}
\end{equation*}
then $A_k(x_{k+1},\ldots, x_r) \geq 0$.

We say that \emph{the pair $(\bf i,m)$ satisfies the condition (P)} if $m_r \geq 0$ and $(\bf i,m)$ satisfies the conditions (P-k) for all $k=1,\ldots,r-1$.
\end{definition}

The following theorem gathers several results on the aforementioned degenerations; 
see~\cite{HY15} for details and original references.

\begin{theorem}\label{thm:GK-Thm}
Let $\bf m\in\mathbb Z_{>0}^r$ and suppose that $({\bf i,m})$ satisfies condition (P). Then
\begin{enumerate}
\item The polytope $P_{\bf i,m}$ is a smooth lattice polytope.

\item 
The symplectic toric manifold $X(P_{\bf i,m})$ associated to the polytope $P_{\bf i,m}$
is the Bott tower with
\begin{equation}\label{ps}
    u_j^0=-e^1_j-\sum_{k>j}^{r}\langle \alpha_{i_k} \alpha_{i_j}^{\vee} \rangle e^1_k \quad \mbox{for all $1\leq j\leq r$}
\end{equation}
and equipped with the ample line bundle
\begin{equation}\label{eq:GK}
     \mathcal M_{\bf i, m}=\sum_{{j=1}}^{{r}}a_j[D_j^0]\quad\mbox{ with }
 a_j=\langle m_{j}\varpi_{i_{j}}+\cdots +m_r\varpi_{i_{r}}, \alpha_{i_j}^{\vee} \rangle .
 \end{equation}
\end{enumerate}
\end{theorem}

\begin{theorem}[{\cite[Theorem 3.4]{HY}}]\label{thm:HY-degen}
Let $\bf m\in\mathbb Z_{>0}^r$ and suppose that $({\bf i,m})$ satisfies condition (P).
Then the polytope $P_{\bf i,m}$ is a Newton-Okounkov body.
In particular, there exists a one-parameter flat family with generic fiber being isomorphic to $Z_{\bf i}$ and 
special fiber isomorphic to the toric variety $X(P_{\bf i,m})$. 
\end{theorem}

\begin{remark}
The polytope $P_{\bf i,m}$ enjoys further nice properties.
For instance, as proved in~\cite{HY}, under the condition $(P)$ the polytope $P_{\bf i,m}$ coincides with the generalized string polytope introduced in~\cite{Fujita}; 
the latter turns out to be unimodular to the polytope $\Delta_{\bf i,m}$ we are considering in Section~\ref{sec:lower} (by~\cite[Theorem 8.2]{Fujita} along with Remark~\ref{rem:Fujita}). 
\end{remark}

Finally, we apply Theorem~\ref{thm:HLD} to compute the Gromov widths of the polarized Bott-Samelson varieties 
$(Z_{\bf i},\mathcal L_{\bf i,m})$ whenever the pair $(\bf i,m)$ satisfies the condition (P).

\begin{corollary}\label{cor:caseline}
Let $\bf m\in\mathbb Z_{>0}^r$ and $({\bf i,m})$ satisfy condition (P). 
Let $X(P_{\bf i,m})$ be the symplectic projective toric manifold associated to the polytope $P_{\bf i,m}$.
Then
\begin{eqnarray*}
\omega_G(Z_{\bf i},\mathcal{L}_{\bf i,m}) & = &\omega_G (X(P_{\bf i,m}))\\
& = & \min\{m_j: \langle\alpha_{i_k}, \alpha_{i_j}^{\vee} \rangle=0 ~ \forall~ k>j\}\\
& = & \min\{\mathcal{L}_{\bf i,m}\cdot C_j: \mbox{$C_j$ line of $Z_{\bf i}$}\}.
\end{eqnarray*}
\end{corollary}

\begin{proof}
Note that the toric degeneration is smooth since it is a Bott manifold by Theorem~\ref{thm:GK-Thm}.
By Theorem~\ref{thm:HY-degen} along with the recalls made in the proof of Corollary~\ref{cor:Kaveh},
 $(Z_{\bf i},\mathcal{L}_{\bf i,m})$ and $X(P_{\bf i,m})$ are symplectomorphic. The first equality thus follows.

To prove the second equality, we apply Theorem~\ref{thm:HLD}.
We thus consider the primitive collections $\mathfrak P_j=\{e^1_j, u_j^0\}$ with $e^1_j+u_j^0=0$.
 By (\ref{ps}), it is clear that 
 \begin{equation} \label{pseq}
     e^1_j+u_j^0=0 \quad \mbox{if and only if }\quad \langle \alpha_{i_k}, \alpha_{i_j}^{\vee} \rangle=0 ~ \forall~ k>j.
 \end{equation}
Let $\widetilde{C}_j\in \mathcal K_j$ be a curve of $X(P_{\bf i,m})$ defined by the primitive relation $e^1_j+u_j^0=0$. 
Then by \cite[Proposition 6.4.1]{cox},  $\mathcal M_{\bf i, m}\cdot \widetilde{C}_j=a_j$ and in turn, $\mathcal M_{\bf i, m}\cdot \widetilde{C}_j=m_j$  thanks to (\ref{eq:GK}) and  (\ref{pseq}). 

The curve $C_j$ is a line if and only if $\mathcal L_k\cdot C_j=0$ for all $j<k$ and this is equivalent to the assertion that $s_{i_j}$ commutes with $s_{i_{j+1}},\ldots,s_{i_r}$; see Remark 4.2 in~\cite{BK21}.
Besides, by Proposition~\ref{prop:Tcurves}(3), $\mathcal L_k\cdot C_j=0$ for all $j>k$.
It follows $\mathcal L_{\bf i, \bf m}\cdot C_j=m_j$ if $C_j$ is a line of $Z_{\bf i}$. 
This yields the last equality and concludes the proof.
\end{proof}

The following example shows that the equalities in Corollary~\ref{cor:caseline} may not hold when the condition (P) is not satisfied.

\begin{example}\label{exple}
Let $G=SL_3$. Take ${\bf i}=(1, 2, 1)$ and  ${\bf m}=(m_1, m_2, m_3)\in \mathbb Z^3_{>0}$ with $m_1+m_2<m_3$.
Then $({\bf i, m})$ does not satisfy the condition (P) since (P-3) does not hold for $(x_2,x_3)=(0,m_3)$.
Besides, $\ell_1=m_1+m_2$, $\ell_2=m_2+m_3$ and $\ell_3=m_3$ hence $\omega_G(Z_{\bf i},\mathcal{L}_{\bf i,m})=\min \{\ell_j: j=1,2,3\}=m_1+m_2$ by Corollary~\ref{cor:main}. 
\end{example}

\subsection*{Acknowledgements} The authors thank the referees for several useful suggestions and comments, which improved the paper.

\noindent
\emph{The authors declare that they have no conflict of interest.}


\end{document}